\documentclass[12pt]{amsart}
\usepackage{geometry}
\usepackage{amsmath,amssymb}
\usepackage{color}
\usepackage{stmaryrd}
\newtheorem{thm}{Theorem}[section]
\newtheorem{cor}[thm]{Corollary}
\newtheorem{lem}[thm]{Lemma}

\newtheorem{prop}[thm]{Proposition}

\newtheorem{exa}[thm]{Example}

\theoremstyle{definition}

\newtheorem{rem}[thm]{Remark}

\numberwithin{equation}{section}

\begin{document}
\title[$2$-nil-sum property]{When is  every non central-unit a sum of two nilpotents?}
\author[Breaz and Zhou]{Simion Breaz and Yiqiang Zhou}

\begin{abstract} A ring is said to satisfy the $2$-nil-sum property if every  non central-unit is the sum of two nilpotents.  We prove that a ring satisfies the $2$-nil-sum property iff it is either a simple ring with the $2$-nil-sum property or a commutative local ring with nil Jacobson radical, and we provide an example of a simple rings with the $2$-nil-sum property that is not commutative. Moreover, a simple right Goldie ring has the $2$-nil-sum property iff it is a field.
 \end{abstract}

\address{Department of Mathematics, Faculty of Mathematics and Computer Science, Babes-Bolyai University, Ro-400084 Cluj-Napoca, Romania}
\email{bodo@math.ubbcluj.ro}

\address{Department of Mathematics and Statistics, Memorial University of Newfoundland, St.John's, NL A1C 5S7, Canada}
\email{zhou@mun.ca}
\thanks{}

\maketitle

\baselineskip=20pt

\bigskip

\section{Introduction}

Throughout, rings are  associative with identity. 
We start by recalling three special types of elements in a ring.  An element in a ring is $2$-good if it is a sum of two units (see \cite{V}), is fine if it is a sum of a unit and a nilpotent (see \cite{CL}), and is $2$-nilgood 
if it is a sum of two nilpotents (see \cite{BC}). By the terminology of V\'{a}mos \cite{V}, a ring is $2$-good if each element is $2$-good.  The study of $2$-good rings  (also called rings with the $2$-sum property in the literature),
initiated by Wolfson \cite{W53} and Zelinsky \cite{Z54},  has attracted considerable interest (see, for example, \cite{BFT}, \cite{KS}, 
\cite{M06}, \cite{S10}, \cite{V} and the references there). A ring is called a fine ring if each nonzero element is fine (the zero element being fine implies that the ring is trivial). Fine rings were introduced and extensively investigated by C\u{a}lug\u{a}reanu and Lam \cite{CL}.

Motivated by the notions of $2$-good rings and fine rings,  we define a ring to satisfy the \textit{$2$-nil-sum property} if every element that is not a central unit is $2$-nilgood (a central unit in a ring being $2$-nilgood implies that the ring is  trivial).  These rings belong to a larger class of rings for which every element is either a unit or a sum of two units. 
But, we will see that the $2$-nil-sum property is very restrictive.  For instance, a ring without nonzero nilpotents satisfies the $2$-nil-sum property iff it is a field.   The rings which is additively generated by nilpotents  have been studied by several authors with connections to rings additively generated by commutators. We refer to \cite{Che} and \cite{Har} for more details. As far as we are  aware, the $2$-nil-sum property of rings has not been discussed in the literature.  Here a study of this topic is conducted. The main results proved here are the following: a ring satisfies the $2$-nil-sum property iff it is either a simple ring with the $2$-nil-sum property or a commutative local ring with nil Jacobson radical (Theorem \ref{reduction}); there exists simple rings with the $2$-nil-sum property that are not commutative (Example \ref{example}); a simple right Goldie ring has the $2$-nil-sum property iff it is a field (Theorem \ref{Goldie}).

For a ring $R$, we denote by $C(R)$, $J(R)$, $U(R)$, and ${\rm nil}(R)$ the center, the Jacobson radical, the group of units of $R$, and the set of nilpotents of $R$, respectively.  We write $\mathbb{M}_{n}(R)$ for the ring of  $n\times n$ matrices over $R$ whose identity is denoted by $I_n$.  
By a \textit{non central-unit} we mean an element that is not a central unit. For $i,j\in \{1,2,\ldots,n\}$, we denote by $E_{ij}$ the $n\times n$ matrix whose $(i,j)$-entry is $1$ and all other entries are $0$.

\section{Basic properties and the reduction theorem}

We start by recording some basic properties.

\begin{lem} \label{lem2.2}Let $R$ be a ring with the $2$-nil-sum property. The following hold:
	\begin{enumerate}
	\item $C(R)\backslash U(R)\subseteq {\rm nil}(R)$. Consequently,  $C(R)$ is a local ring with nil Jacobson radical.
\item For every proper ideal $I$ of $R$ we have $I\subseteq C(R)\cap {\rm nil}(R)$.
			\item All proper ideals of $R$ are contained in $J(R)$. 
		\item $J(R)$ is nil and $J(R)\subseteq C(R)$.
		%\item $C(R)$ is a local ring with nil Jacobson radical.
		\item For any ideal $I$ of $R$, central units of $R/I$ can be lifted to central units of $R$. 
	%	\item $R/I$ satisfies the $2$-nil-sum property for every ideal $I$ of $R$.
	\end{enumerate}
\end{lem}
\begin{proof}
(1) Assume that $r$ is a central element that is not a unit. Then $r=b_1+b_2$, where $b_1, b_2\in {\rm nil}(R)$. Since $r$ is central, $b_1$ and $b_2$  commute.  So $r=b_1+b_2\in {\rm nil}(R)$.

(2) Assume that $I\backslash  C(R)\neq \varnothing$.  Let $a\in  I\backslash C(R)$. Then $1+a$ is a sum of two nilpotents, so $\bar 1\in R/I$ is a sum of two nilpotents,  a contradiction.  Hence $I\subseteq C(R)\backslash U(R)$, so $I\subseteq {\rm nil}(R)$ by (1).

(3) and (4)  are clear by (2).

%	$(1)$ If $e\not= 1$ is a central idempotent, then  $e=b_1+b_2$, where $b_1, b_2\in {\rm nil}(R)$. Since $e$ is central, $b_1$ and $b_2$  commute.  So $b_1+b_2\in {\rm nil}(R)$, and hence $e=0$.
	
	% $(2)$ Let $I$ be a proper ideal of $R$. If $a\in I$ is not contained in $J(R)$, then $1+ar$ is not a unit for some $r\in R$, so $1+ar=b_1+b_2$ where $b_1,b_2$ are nilpotent. Then, in $R/I$, $\bar 1=\bar b_1+\bar b_2$, a sum of two nilpotents. This is a contradiction. So $I\subseteq J(R)$. Hence $(2)$ holds.

	% $(3)$ Let $j\in J(R)$. Assume that $j\notin C(R)$. Then $1+j=b_1+b_2$, where $b_1, b_2\in {\rm nil}(R)$. Then, in $R/J(R)$, $\bar 1=\bar  b_1+\bar b_2$, a sum of two nilpotents. This is a contradiction. So $J(R)\subseteq C(R)$. Thus, each $j\in J(R)$ is a sum of two commuting nilpotents, so  $j$ is nilpotent. 

	%Every nonunit $a$ in $C(R)$ is a nonunit in $R$, so is a sum of two commuting nilpotents. Hence $a$ is nilpotent. So $(4)$ holds.

	 $(5)$  Let $\bar u:=u+I$ be a central unit of $R/I$. We claim that $u\in R$ is a central unit. Otherwise,  $u$ is a sum of two nilpotents in $R$. Hence $\bar u$ is a sum of two nilpotents in $R/I$, a contradiction.
%	 
%	 $(6)$  If $\bar a:=a+I$ is a non central-unit of $R/I$, then $a$ is a non central-unit of $R$. So $a$ is a sum of two nilpotents in $R$. Hence $\bar a$ is a sum of two nilpotents in $R/I$.
	\end{proof}
	
\begin{thm}\label{reduction} {\rm (}Reduction Theorem{\rm )} A ring $R$ satisfies the $2$-nil-sum property iff $R$ is either a simple ring with the $2$-nil-sum property or a commutative local ring with nil Jacobson radical. 
\end{thm}

%\noindent {\bf Proof of Theorem 1.1}.  
\begin{proof}
The sufficiency is clear. For the necessity, suppose that $R$ satisfies the $2$-nil-sum property but $R$ is not simple. 	Then, by Lemma \ref{lem2.2}(3), $J(R)\not= 0$. Let $0\not=x\in J(R)$. If $a,b\in R$ then, noting $J(R)\subseteq C(R)$, we have  $$(ab)x=a(bx)=(bx)a=b(xa)=b(ax)=(ba)x,$$ hence $(ab-ba)x=0$. It follows that $ab-ba$ belongs to the ideal
$Ann(x)=\{r\in R \mid rx=0\}$. Since $Ann(x)\not= R$, $Ann(x)\subseteq J(R)$. % by Lemma \ref{lem2.2}(2).
It follows that $R/J(R)$ is a commutative ring with the $2$-nil-sum property, and hence is local by Lemma 2.1(1).  It follows that $R/J(R)$ is a field. Thus, by Lemma \ref{lem2.2}(5),  for any $a\in R\backslash J(R)$, $a=u+j$ where $u$ is a central unit and $j\in J(R)$. It follows that $a$ is central. Hence, $R$ is commutative, and so $R$ is a commutative local ring with $J(R)$ nil.  
\end{proof}

\bigskip
%It is clear by Theorem 1.1 that every non central-unit of a ring $R$ is a sum of two commuting nilpotents iff $R$ is a commutative local ring with $J(R)$ nil. 

If $R$ is a commutative ring with the $2$-nil-sum property then every non central-unit $a\in R$ is nilpotent, so it is a sum of the form $a=b+c$ with $b^1=0$ and $c=a$ is nilpotent.  Thus, $R$ is a ring with the $2$-nil-sum property of type $(1, \infty)$ as defined below.  The $2$-nil-sum property is refined as follows: let $p,q$ be integers with $p\ge 1$ and $q\ge 1$. We say that a ring $R$ has the $2$-nil-sum property of type $(p,q)$ if, for each non central-unit $a$ in $R$, $a=b+c$ where $b^p=c^q=0$. The ring $R$ has the $2$-nil-sum property of type $(p,\infty)$ if, for each non central-unit $a$ in $R$, $a=b+c$ where $b^p=0$ and $c$ is nilpotent. Clearly,  a ring $R$ is a field iff $R$ has the $2$-nil-sum property of type $(1,1)$.

\begin{prop} \label{lem5.2} Suppose that $R$ has the  $2$-nil-sum property of type $(2,\infty)$. Then $R$ is commutative.
\end{prop}
\begin{proof} We observe that $R$ has no non-trivial idempotents:  if $1\not= e^2=e\in R$ is a sum of  a square-zero element and a nilpotent it follows   by \cite[Proposition 2]{FPS} that  $e=0$.

Suppose that $R$ is not commutative. Then $R$ is a simple ring by Theorem 2.2, hence $0$ is the only central nilpotent.  But ${\rm nil}(R)\not= 0$, hence there exists a square-zero element $x$ that is not central. Since $R$ is of type $(2,\infty)$, there exist $y,z\in R$ such that 
	$$1-x=y+z,$$ 
	where $y^2=0$ and $z^n=0$ for some $n\ge 1$. Then $(1-x-y)^n=0$. It follows that
	\begin{equation*}
	1=n_1(x+y)+n_2(xy+yx)+n_3(xyx+yxy)+n_4(xyxy+yxyx)+
	\cdots,
	\end{equation*}
	where $n_1=n$ and each $n_i$ is an integer. Multiplying both sides of this equation from the left by $x$ gives
	\begin{equation*}
	\begin{split}
	x&=n_1xy+n_2xyx+n_3xyxy+n_4xyxyx+\cdots\\
	\end{split}
	\end{equation*}	
	Multiplying both sides of this equality from the right by $y$ gives
	\begin{equation*}
	%\begin{split}
	xy=n_2xyxy+n_4xyxyxy+\cdots =a(xy)^2=(xy)^2a,
	%\end{split}
	\end{equation*}
	where $a=n_2+n_4xy+\cdots$. It follows that $a(xy)$ is an idempotent.  Since $x^2=0$, we have $a(xy)\not= 1$. We obtain  $a(xy)=0$, hence $xy=0$. It follows that $x=0$, a contradiction.  
\end{proof}

\begin{prop} Suppose that $R$ is of characteristic 0 and has the  $2$-nil-sum property of type $(p,q)$ where $p\le 3$ and $q\le 5$. Then $R$ is commutative.
\end{prop}
\begin{proof} We can assume that $p=3$ and $q=5$. Suppose that $R$ is not commutative. As in the proof of Proposition \ref{lem5.2}, there exists a square-zero element $x$ that is not central.  So, there exist $y,z\in R$ such that $1-x=y+z$ and $y^3=z^5=0$. This contradicts \cite[Proposition 9]{S}. 
\end{proof}

We close this section with the following 

\noindent\textbf{Open Question.} If a ring $R$ has the $2$-nil-sum property of type $(p,q)$ with $p\in\mathbb{N}$ and $q\in\mathbb{N}\cup\{\infty\}$, is it commutative?

\section{Simple rings}

\begin{exa} \label{example}
There exists a simple ring with the $2$-nil-sum property that is not commutative.
\end{exa}

\begin{proof}
Let $\mathbb{F}_2$ be the field of $2$ elements. For each integer $n\ge 0$ we consider the diagonal morphism $$\varphi_n:\mathbb{M}_{2^n}(\mathbb{F}_2)\to \mathbb{M}_{2^{n+1}}(\mathbb{F}_2), A\mapsto \mathrm{diag}(A,A).$$ This is an inductive system, and we denote its colimit by  $R$, while $\varphi_n:\mathbb{M}_{2^n}(\mathbb{F}_2)\to R$ denote the canonical ring morphisms. Then $R$ is a simple ring (see \cite[Example 8.1]{G76}). 

Let $0\neq r\in R$ a non central-unit of $R$. Then we can view it as an image $r=\varphi_n(A)$ for some $0\neq A\in \mathbb{M}_{2^n}(\mathbb{F}_2)$. Note that $A\neq I_{2^n}$ since otherwise $r$ is the identity of $R$.
Observe that $r=\varphi_{n+1}(\mathrm{diag}(A,A))$. From $A\neq I_{2^n}$ it follows that $\mathrm{diag}(A,A)$ is not a scalar matrix. Since its trace is zero, there exist two nilpotent matrices $N_1$ and $N_2$ in $\mathbb{M}_{2^{n+1}}(\mathbb{F}_2)$ such that $\mathrm{diag}(A,A)=N_1+N_2$, see \cite[Proposition 3(ii)]{BC}. Then $r=\varphi_{n+1}(N_1)+\varphi_{n+1}(N_2)$ is a sum of two nilpotents.
\end{proof}

Similar examples can be obtained by using other fields (of positive characteristic).  We do not know other kind of examples of non-commutative rings with the $2$-nil-sum property.  In fact,  we will prove that if a simple ring with the $2$-nil-sum property satisfies some standard finiteness conditions then it is a field.  A basic example is the following.

\begin{exa}
{\rm Let $F$ be a field and $n\geq 2$ an integer. Since the traces of all nilpotent matrices in $\mathbb{M}_n(F)$ are $0$, it follows that $\mathbb{M}_n(F)$ does not have the $2$-nil-sum property. This can be easily extended to matrices over commutative rings (since in this case the trace of a nilpotent matrix has to be nilpotent).}
\end{exa}

Below we will study matrix rings over division rings. In this case, the above argument does not work because nilpotent matrices over division rings may have non-zero traces and,   moreover,  the trace is no longer invariant under similarity. For instance, if $A=\begin{pmatrix}\bf i&\bf j\\
-\bf j&\bf i\end{pmatrix}\in {\mathbb M}_2(\mathbb H)$, where  $\mathbb H$ is the ring of real quaternions, then $A^2=0$ and $\begin{pmatrix}1&0\\
\bf k&1\end{pmatrix}A\begin{pmatrix}1&0\\
-\bf k&1\end{pmatrix}=\begin{pmatrix}0&\bf j\\
0&0\end{pmatrix}$. %But ${\rm tr}(A)=2\bf i\not= 0$.   

We start with some elementary lemmas. We include the details of the proofs for the reader's convenience.

\begin{lem}\label{nil}
	Let $D$ be a division ring. If $A\in {\mathbb M}_n(D)$ is nilpotent, then $A^n=0$.
\end{lem}
\begin{proof} We identify $A$ with a linear transformation of the right vector space $V:=D^n$ over $D$. 
Then the chain
\begin{equation*}
Im(A)\supset Im(A^2)\supset Im(A^3)\supset \cdots
\end{equation*}	
should be strictly decreasing (otherwise, $A$ is not nilpotent). Since 	${\rm dim}(V_D)=n$, it follows that $A^n=0$. 
	\end{proof}

\begin{lem}\label{lem2.3}
	Let $R$ be a ring and $n\ge 2$. If  $A=\begin{pmatrix}a_{11}&\cdots&\cdots&a_{1n}\\
	1&\ddots&\ddots&\vdots\\
	&\ddots&\ddots&\vdots\\
	&&1&a_{nn}\end{pmatrix}\in {\mathbb M}_n(R)$, then $$A^k=\left(\begin{array}{c|c} * & * \\ \hline
	\begin{matrix}\sum_{i=1}^ka_{ii}&\cdots&\cdots&\cdots\\
	1&\ddots&\ddots&\vdots\\
	&\ddots&\ddots&\vdots\\
	&&1&\sum_{i=n-k+1}^na_{ii}\end{matrix} & * \end{array}\right),$$
	where the lower left block has size $(n-k+1)\times (n-k+1)$. So the $(k,1)$-entry of $A^k$ is $a_{11}+\cdots+a_{kk}$ for $k=1,\ldots, n$; in particular, the $(n,1)$-entry of $A^n$ is $a_{11}+\dots+a_{nn}$. 
\end{lem}
\begin{proof}
	The claim can be proved by induction on $k$. 
\end{proof}
\begin{lem}\label{lem2.4}
Let $D$ be a division ring and $n\geq 2$ an integer. 
If $X=(x_{ij})\in {\mathbb M}_n(D)$ is a matrix such that the $(n-1)$-dimensional column vector $\beta=\left[ 
	\begin{array}{c}
	x_{21}\\
		\vdots\\
	x_{n1}\\
	\end{array}
	\right] $ is not zero, then there exists an invertible matrix $U=\begin{pmatrix}1&\bf 0\\
	\bf 0&V\end{pmatrix}\in {\mathbb M}_n(D)$ with $V\in {\mathbb M}_{n-1}(D)$, and $1< k\leq n$ such that 
	$$UXU^{-1}=\left(\begin{array}{c|c} \begin{matrix}w_{11}&\cdots&\cdots&w_{1k}\\
	1&\ddots&\ddots&\vdots\\
	&\ddots&\ddots&\vdots\\
	&&1&w_{kk} \end{matrix} & Y  \\ \hline
	\bf 0&  Z \end{array}\right).$$
\end{lem}

\begin{proof}
Write $X=\begin{pmatrix}x_{11}&Y_1 \\
	W_1&Z_1\end{pmatrix}$ in block form where $Z_1$ is an $(n-1)\times (n-1)$ matrix. Since $W_1=\beta\not= 0$ there exists an invertible matrix $V_1$ in ${\mathbb M}_{n-1}(D)$ such that $V_1W_1=\left[ 
	\begin{array}{c}
	1\\ 
	0\\ 
	\vdots\\
	0\\
	\end{array}
	\right]$. Then $U_1=\begin{pmatrix}1&\bf 0\\
	\bf 0&V_1\end{pmatrix}$ is invertible in ${\mathbb M}_n(D)$ 
	and 
	$U_1XU_1^{-1}=\begin{pmatrix}x_{11}&\gamma_1 V_{1}^{-1}\\
	V_1W_1 &V_1X_1V_{1}^{-1}\end{pmatrix}$, so the first column of $U_1XU_1^{-1}$ is  $\left[ 
	\begin{array}{c}
	x_{11}\\
	1\\ 
	0\\ 
	\vdots\\
	0\\
	\end{array}
	\right]$. 
	
	Write $U_1XU_1^{-1}=\left(\begin{array}{c|c}\begin{matrix}y_{11}&y_{12} \\
	1&y_{22}\end{matrix}&Y_2 \\ \hline
	W_2 &Z_2\end{array}\right)$, where $Z_2$ is an $(n-2)\times (n-2)$ matrix and the first column of $W_2$ is $\bf 0$. If $W_2=\bf 0$ we can simply take $U=U_1$. 
	If  $W_2\not= \bf 0$, there exists an invertible matrix $V_2$ in ${\mathbb M}_{n-2}(D)$ such that $V_2W_2=\left[ 
	\begin{array}{cc}
	0&1\\ 
	0&0\\ 
	\vdots&\vdots\\
	0&0\\
	\end{array}
	\right]$. Observe that $U_2=\begin{pmatrix}I_2&\bf 0\\
	\bf 0&V_2\end{pmatrix}$ is invertible in ${\mathbb M}_n(D)$ and 
	$$U_2U_1XU_1^{-1}U_2^{-1}=\left(\begin{array}{c|c}\begin{matrix}z_{11}&z_{12} &z_{13} \\ 
	1&z_{22} & z_{23} \\
	0&1 & z_{33}\end{matrix}&Y_3\\ \hline
	W_3 &Z_3\end{array}\right).$$
	The block consisting of the first two columns of $U_2U_1XU_1^{-1}U_2^{-1}$ is of the form $\left[ 
	\begin{array}{cc}
	z_{11}&z_{12}\\
	1&z_{22}\\
	0&1\\ 
	0&0\\ 
	\vdots&\vdots\\
	0&0\\
	\end{array}
	\right]$, so the first two columns of $W_3$ are zero. If $W_3=\bf 0$, we can take $U=U_2U_1$. If $W_3\not= \bf 0$, we continue in a similar way. 
	
This process has to stop after a finite number of steps. Therefore, there exists an integer $1< k\leq n$, and a family of invertible matrices $U_l=\begin{pmatrix}I_l&\bf 0\\
	\bf 0&V_l\end{pmatrix}$, $1\leq l\leq k-1$, such that 
\begin{itemize}	
\item $(U_l \cdots  U_1)X(U_l \cdots  U_1)^{-1}=\left(\begin{array}{c|c}\begin{matrix}w_{11}& w_{12} & \dots &w_{1l} & w_{1(l+1)} \\
	1 & w_{22} & \dots &w_{2l} & w_{2(l+1)} \\ 
	\vdots & \vdots & &\vdots & \vdots \\
	0& 0 &\dots &1 &w_{(l+1)(l+1)}\end{matrix}&X_{l+1}\\ \hline
	W_{l+1} &Y_{l+1}\end{array}\right),$
	\item the first $l$ columns of $W_{l+1}$ are zero,
\item $k=n$, or $k<n$ and $W_{k}=\bf 0$.
\end{itemize} 		
Now take $U=U_{k-1} \cdots U_1$. By the constructions of $U_1,\dots, U_{k-1}$, one easily sees that $U$ has the desired property. 	
	\end{proof}

\begin{rem}
The above lemma will be applied for the case when $X$ is nilpotent. We want to point out that it was proved in \cite{Cohn} that every nilpotent matrix is similar to a rational form (as in the commutative case).  But we need the weaker property presented above since we have to use the form of the matrix $U$. 
\end{rem}

\begin{lem}\label{lem-lines} Let $D$ be a division ring and $A=(a_{ij})\in {\mathbb M}_n(D)$ a matrix  with the property that the $k$-th row of $A$ is the only non-zero row of $A$.  Then $A$ is a sum of two nilpotents if and only if 
$a_{kk}=0$.
\end{lem}

\begin{proof} The sufficienty is obvious. For the necessity, let us observe that $A$ is 
similar to $A'=\begin{pmatrix}a&\alpha\\
	\bf 0&\bf 0\end{pmatrix}$,  where  $a=a_{kk}\in D$ and $\alpha$ is an $(n-1)$-dimensional row vector.  Hence  $A'=X+Y$, where $X, Y$ are nilpotents.  
Write $X=\begin{pmatrix}x_{11}&\gamma \\
	\beta&X_1\end{pmatrix}$ in block form where $X_1$ is an $(n-1)\times (n-1)$ matrix.  If $\beta= 0$,  then $a=x_{11}+(a-x_{11})$ is a sum of two nilpotents in $D$; so $a=0$.   Thus,  we can assume that $\beta\not= 0$. There exists an invertible matrix $U=\begin{pmatrix}1&\bf 0\\
	\bf 0&V\end{pmatrix} \in\mathbb{M}_n(D)$ and an integer $k>1$ as in Lemma \ref{lem2.4} such that $$UXU^{-1}=\left(\begin{array}{c|c} \begin{matrix}w_{11}&\cdots&\cdots&w_{1k}\\
	1&\ddots&\ddots&\vdots\\
	&\ddots&\ddots&\vdots\\
	&&1&w_{kk} \end{matrix} & T  \\ \hline
	\bf 0&  Z \end{array}\right).$$  
Note that $UA'U^{-1}=\begin{pmatrix}a&\begin{pmatrix}a_2&\cdots&a_n\end{pmatrix}\\
		\bf 0&\bf 0\end{pmatrix}$ for some $a_2,\ldots,a_n$ in $D$, and that $UXU^{-1}$ and $UYU^{-1}$ are nilpotent matrices.
It follows that the matrices $$\left(\begin{matrix}w_{11}&\cdots&\cdots&w_{1k}\\
	1&\ddots&\ddots&\vdots\\
	&\ddots&\ddots&\vdots\\
	&&1&w_{kk} \end{matrix}\right)\text{ and }\left(\begin{matrix}w_{11}-a&\cdots&\cdots&w_{1k}-a_k\\
	1&\ddots&\ddots&\vdots\\
	&\ddots&\ddots&\vdots\\
	&&1&w_{kk} \end{matrix}\right)$$ are both nilpotent matrices. Using Lemma \ref{nil} and Lemma \ref{lem2.3}, we obtain the equalities $w_{11}+\dots+w_{kk}=0=w_{11}+\dots +w_{kk}-a$, and hence $a=0$.  
\end{proof}

\begin{thm}\label{Goldie}  Let $R$ be a simple right Goldie ring with the $2$-nil-sum property. Then $R$ is a field.
\end{thm}

%\noindent {\bf Proof of Theorem 1.2}.   
\begin{proof}
Note that we can view $R$ as a right order of ${\mathbb M}_n(D)$ where $n\ge 1$ and $D$ is a division ring.  If $n=1$, then $R$ is a subring of a division ring. Since ${\rm nil}(R)=0$, the only non central-unit of $R$ is $0$, so $R$ is a field.  

We next show that $n\ge 2$ yields a contradiction.  
Assume that $n\ge 2$. For all $1\leq k\leq n$, we denote by $\mathcal{A}_k$ the set of all non-zero matrices $A=(a_{ij})\in R$ with the property that the $k$-th row of $A$ is the only non-zero row of $A$.

  Since $R$ is a right order of ${\mathbb M}_n(D)$, for $1\leq k\leq n$ we can write $E_{kk}=U_{k}S_{k}^{-1}$ where $U_{k}, S_{k}\in R$ and where $U_{k}=(u^{(k)}_{ij})$ and $S_{k}=(s^{(k)}_{ij})$. It follows that $U_{k}=E_{kk}S_{k}\in \mathcal{A}_k$, so all sets $\mathcal{A}_k$ are non-empty. 

Moreover, if $1\leq k\leq n$, we denote by $i(k)$ the minimal integer with the property that there exists $(a_{ij})\in \mathcal{A}_k$ such that $a_{kl}=0$ for all $l<i(k)$ and $a_{ki(k)}\neq 0$. 

We show that if $k<i(k)$ then $i(k)<i(i(k))$. In order to prove this, let $A=(a_{ij})\in\mathcal{A}_k$ be a matrix such that  $a_{kl}=0$ for all $l<i(k)$ and
$a:=a_{ki(k)}\neq 0$. For any $B=(b_{ij})\in \mathcal{A}_{i(k)}$, the $k$-th row of $AB$ is $(ab_{i(k)1},\dots,ab_{i(k)i(k)},\dots, ab_{i(k)n})$ and this is the only possible non-zero row of $AB$. Since $a\neq 0$ it follows, by the minimality of $i(k)$, that $b_{i(k)1}=\dots=b_{i(k)(i(k)-1)}=0$. Moreover,  by Lemma \ref{lem-lines} it follows that $b_{i(k)i(k)}=0$ as $B\in \mathcal{A}_{i(k)}$. Hence, $i(k)<i(i(k))$.

Now we close the proof by observing that $1<i(1)$, hence we obtain a strictly increasing infinite sequence bounded above by $n$: $1<i(1)<i(i(1))<\cdots $. This is a contradiction. 	
\end{proof}

%\begin{cor} Let $V_D$ be a non-trivial vector space over a division ring $D$. Then  ${\rm End}(V_D)$ satisfies the $2$-nil-sum property if and only if ${\rm dim}(V)=1$ and $D$ is a field.\end{cor}\begin{proof}The sufficiency is clear. For the necessity, 		the assumption ensures that $R:={\rm End}(V_D)$ is a simple ring by Theorem 1.1,  as $J(R)=0$. So ${\rm dim}(V)<\infty$.  Hence $R\cong {\mathbb M}_n(D)$ where $n=\rm dim(V)$.  It follows from Theorem 1.2 that  ${\rm dim}(V)=1$ and $D$ is a field. 	\end{proof}	

\begin{rem}
By Theorem 3.8,  for a division ring $D$, the ring $\mathbb{M}_n(D)$ satisfies the $2$-nil-sum property iff $n=1$ and $D$ is a field.  Note that there exists a division ring $D$ such that all matrices in ${\mathbb M}_3(D)$ are sums of three nilpotent matrices (see \cite{Salwa-02}).
\end{rem}

%\begin{cor}\label{cor2.4} If $R$ is a $2$-nilgood ring and ${\rm nil}(R)$ is an ideal, then  $R$ is a commutative local ring with $J(R)$ nil.\end{cor}\begin{proof} The assumption implies that all elements in $R\backslash {\rm nil}(R)$ are central units. Moreover, ${\rm nil}(R)=J(R)$ by Lemma \ref{lem2.2}(3).\end{proof}		

%\begin{rem}Almost all elements in a ring $R$ with the $2$-nil-sum property are fine: if $x-1\in R$ is not a central unit, then there exist two nilpotents $b_1$ and $b_2$ such that $x=(1+b_1)+b_2$, hence $x$ is fine. However, each central nilpotent element is not fine. \end{rem}

%For an $n\times n$  nilpotent matrix $B$ over a field $F$, $A$ is similar to a strictly upper triangular matrix, so ${\rm tr}(A)=0$ because the trace is a similarity-invariant. \big(The characteristic polynomial of $A$ being $x^n$ also implies that ${\rm tr}(A)=0$.\big) Thus, a $2$-nilgood matrix in ${\mathbb M}_n(F)$  ($n\ge 2$) must have trace $0$, and hence ${\mathbb M}_n(F)$ is not $2$-nilgood. This argument does not work for matrix rings over a division ring, because the trace is no longer a similarity-invariant in this situation.
	 
%\begin{exa} \label{exa3.1}Let $A=\begin{pmatrix}\bf i&\bf j\\-\bf j&\bf i\end{pmatrix}\in {\mathbb M}_2(\mathbb H)$, where  $\mathbb H$ is the ring of real quaternions. Then $A^2=0$ and $\begin{pmatrix}1&0\\\bf k&1\end{pmatrix}A\begin{pmatrix}1&0\\-\bf k&1\end{pmatrix}=\begin{pmatrix}0&\bf j\\0&0\end{pmatrix}$. But ${\rm tr}(A)=2\bf i\not= 0$.\end{exa}	

We close the paper, with some applications of Theorem \ref{reduction} and Theorem \ref{Goldie}. 
A nonzero right ideal $I$ of a ring $R$ is called a uniform right ideal if the intersection of any two nonzero right ideals contained in $I$ is nonzero. A ring $R$ is semipotent if every nonzero right ideal not contained in $J(R)$ contains a nonzero idempotent.   A ring $R$ is said to be of bounded index (of nilpotence) if there is a positive integer $n$ such that $a^n=0$ for all nilpotents $a$ of $R$.

\begin{cor}\label{cor3.8}
	A ring $R$ with the $2$-nil-sum property is commutative under any of the following additional assumptions:
		\begin{enumerate}
			\item  $R$ contains a uniform ideal;
			\item 	$R$ is a semipotent ring of bounded index;
			\item $R$ is a semilocal ring.
		\end{enumerate}
	\end{cor}
	
\begin{proof} By Theorem \ref{reduction},  suppose that $R$ is a simple ring.

(1) It follows from \cite{H} that $R$ contains no infinite direct sums of right ideals (i.e., $R_R$ is of finite uniform dimension).
	 Moreover, from the last remark of \cite{K} it follows that the maximal right ring of quotients of $R$ is simple artinian.   Thus, $R$ is right Goldie by \cite[Proposition 13.41]{L}.

(2) If $R$ is semipotent of bounded index, then $R$ is simple artinian by \cite[Corollary 6]{B}.

(3) If $R$ is semilocal, then $R$ is simple artinian.

So the claim follows from Theorem \ref{Goldie}.
	\end{proof}

\section*{Acknowledgment} 
Breaz was supported by a grant of the Ministry of Research, Innovation and Digitization, CNCS/CCCDI – UEFISCDI, project number PN-III-P4-ID-PCE-2020-0454, within PNCDI III.   Zhou was supported by a Discovery Grant
(RGPIN-2016-04706) from NSERC of Canada.

\end{document}